\documentclass[11pt,a4paper]{article}
\usepackage{amsmath,amsthm,amsfonts,amssymb}
\usepackage{lscape}
\usepackage{float}
\usepackage[caption = false]{subfig}
\usepackage[demo]{graphicx}
\usepackage{epstopdf}
\usepackage[mathscr]{euscript}
\usepackage[utf8]{inputenc}
\usepackage{authblk}
\usepackage{enumerate}
\usepackage{hyperref}

\newcommand{\C}{{\mathbb C}}
\newtheorem{thm}{Theorem}[section]

\newtheorem{lem}{Lemma}[section]
\newtheorem{ex}{Example}
\newtheorem{defn}{Definition}
\newtheorem{rem}{Remark}[section]
\newtheorem{obs}[thm]{Observation}

\begin{document}
\title{On Fatou sets containing Baker omitted value}

\author[1]{Subhasis Ghora\footnote{sg36@iitbbs.ac.in} }
\author[1]{Tarakanta Nayak\footnote{tnayak@iitbbs.ac.in(Corresponding author)} }
\author[2]{ Satyajit Sahoo\footnote{satyajitsahoo2010@gmail.com }}

\affil[1]{\textit{School of Basic Sciences, 
	 Indian Institute of Technology Bhubaneswar, India}  }
 
\affil[2]{\textit{P.G. Department of Mathematics, Utkal University, 	India}  }
\date{}
\maketitle
\begin{abstract}
 An omitted value of a transcendental meromorphic function $f$ is called a Baker omitted value, in short \textit{bov} if there is a disk $D$ centered at the bov such that each component of the boundary of $f^{-1}(D)$ is bounded.  Assuming that the bov is in the Fatou set of $f$, this article investigates the dynamics of the function. Firstly, the connectivity of all the Fatou components are determined. If $U$ is the Fatou component containing the bov then it is proved that a Fatou component $U'$ is infinitely connected if and only if it lands on $U$, i.e. $f^{k}(U') \subset U$
  for some $k \geq 1$. Every other Fatou component is either simply connected or lands on a Herman ring. Further, assuming that the number of critical points in the Fatou set whose forward orbits do not intersect $U$ is finite, we have shown that the connectivity of each  Fatou component belongs to a finite set. This set is independent of the Fatou components. It is proved that the Fatou component containing the bov is completely invariant whenever it is forward invariant. Further, if the invariant Fatou component is an attracting domain and compactly contains all the critical values of the function  then the Julia set is totally disconnected.  Baker domains are shown to be non-existent whenever the bov is in the Fatou set. It is also proved that, if there is a $2$-periodic Baker domain (these are not ruled out when the bov is in the Julia set), or a $2$-periodic attracting or parabolic domain containing the bov then the function has no Herman ring. Some examples exhibiting different possibilities for the Fatou set are discussed. This includes the first example of a meromorphic function with an omitted value which has two infinitely connected Fatou components. 
\end{abstract}

Mathematics Subject Classification(2010) 37F10,  37F45

\section{Introduction}
Let $\widehat{\mathbb{C}}=\mathbb{C} \bigcup \{\infty\}$ and $f:\mathbb{C}\rightarrow \widehat{\mathbb{C}}$ be a transcendental meromorphic function with a single essential singularity (which we choose to be at $\infty$) such that it has either at least two poles or exactly one pole which is not an omitted value. Such maps are known as general meromorphic functions. The $n$-th iterate of $f$ is denoted by $f^{n}$. The notion of normality of the family of functions  $\{f^{n}\}_{n>0}$ gives rise to a partition of the Riemann sphere into two sets, namely the Fatou set and the Julia set. The Fatou set of $f$, denoted by $\mathcal{F}(f)$ is the set of all points where the family of functions  $\{f^{n}\}_{n>0}$ is well-defined and in a neighbourhood of which $\{f^{n}\}_{n>0}$ is normal. This set is  also known as the stable set. The complement of the Fatou set is the Julia set, and is denoted by $\mathcal{J}(f)$. A maximally connected subset of the Fatou set is called a {Fatou component}. For a {Fatou component} $V$, $V_{k}$ denotes the {Fatou component} containing $f^{k}(V)$ for $k \geq 0$ where $V_0=V$. The connectivity of $V$, denoted by $c(V)$ is the number of components of $\widehat{\mathbb{C}} \setminus V$. We say $V$ is infinitely connected if $c(V)=\infty$. A {Fatou component} $V$ is called $p$-periodic if $p$ is the least natural number satisfying $V_p = V $. We say $V$ is invariant if $p=1$. An invariant Fatou component $V$ is called completely invariant if it is backward invariant (i.e. $f^{-1}(V) \subseteq V$). A periodic Fatou component is one of the five types, namely an attracting domain,  a parabolic domain, a Siegel disk, a Herman ring or a Baker domain. 
A $p$-periodic Fatou component $V$ is called an attracting domain or a parabolic domain if $\{f^{np}\}_{n>0}$ converges uniformly on $V$ to a $p$-periodic  attracting or parabolic point  respectively. 
 It is called a Herman ring (or a Siegel disk) if there exists an analytic homeomorphism $\phi : V \rightarrow A_r =\{z:1<|z|<r\}$ ( or $\phi : V \rightarrow D_r =\{z:|z|<r\}$) such that  $\phi(f^p(\phi^{-1}(z)))=e^{i 2 \pi \alpha}z$  for all $z \in A_r$ (or $D_r$ respectively) and for some irrational number $\alpha$.  A $p$-periodic Baker domain is one on which $\{f^{np}\}_{n>0}$ converges uniformly to a point where $f^p$ is not well-defined. If $V$ is not periodic but $V_n$ is periodic for some natural number $n$, then $V$ is called pre-periodic.  If a Fatou component is neither periodic nor pre-periodic, then it is called wandering. Further details can be found in ~\cite{ber93}.
	\par 
 A critical value is the image of a critical point, that is, $f(z_{0})\ \mbox{where}\ f'(z_{0})=0$.  The smallest natural number $k$ for which $f^{(k)}(z_0) \neq 0$ is called the local degree of $f$ at $z_0$. In this case, the local behaviour of $f$ at $z_0$ is like  $z \mapsto z^k$. A multiple pole is also considered as a critical point. The local degree of $f$ at a non-critical point is one. A point $a\in \widehat{\mathbb{C}}$ is an asymptotic value of $f$ if there exists a curve $\gamma:[0,\infty)\rightarrow \mathbb{C}$ with $\lim_{t\rightarrow \infty} \gamma(t)=\infty$ such that $ \lim_{t\rightarrow \infty} f(\gamma(t))=a$.  The set of singular values is the closure of all critical values and asymptotic values of $f$. 
 
 \par 
  A point $z_{0}\in  \mathbb{C}$ is said to be an omitted value of a function $f$ if $f(z)\neq z_{0}$ for any $z\in\mathbb{C}$.  It is well-known that there can be at most two omitted values for a meromorphic function and that each omitted value is an asymptotic value.
 Our concern is a special type of omitted value.
\begin{defn}
	
	An omitted value $b \in  \mathbb{C}$ of a meromorphic function $f$ is said to be a Baker omitted value,  in short \textit{bov} if there is a disk $D$ with center at $b$ such that each component of the boundary of $f^{-1}(D)$ is bounded.

\end{defn} 

   If $a \in \widehat{\C}$ is an asymptotic value of $f$ and   $D_r(a)$ is a disk (with respect to the spherical
metric) centered at $a$ and with radius $r$ then a component $V_r$ of $f^{-1}(D_r(a))$ can be chosen in such a way that
$V_{r_1}\subset V_{r_2}$ for $0<r_1<r_2$. In this case, $\bigcap_{r>0} V_r= \emptyset$ and the choice $r \mapsto V_r$ defines a transcendental singularity. It is usual to say that a singularity $V$ lies over $a$. The singularity $V$ lying over $a$ is called direct if there exists $r>0$ such that $f(z) \neq a$ for any $z \in V_r$. Every singularity lying over an omitted value  is always direct and there can be more than one singularity lying over an omitted value. A direct singularity is called logarithmic if  $f: V_r \to D_{r}(a) \setminus \{a\}$ is a universal covering. In this case $V_r$ is simply connected. The dynamics of functions with an asymptotic (in particular, omitted) value over which there is a logarithmic singularity is reasonably well-studied (See for example, ~\cite{berg08,tanz,ez}). The singularity lying over a bov is not logarithmic even though it is direct. There is only one singularity lying over a bov. These facts are evident from Lemma~\ref{bov-implication1}. The bov is always a limit point of critical values (Lemma~\ref{critically infinite}) giving that every meromorphic function with a bov have infinitely many singular values. In view of a large body of literature on the dynamics of functions of finite type (those with only finitely many singular values) or with a direct singularity, the  study of dynamics of functions with bov is essentially a new direction.  
  
\par  The singular values are known to control the dynamics-the Fatou and the Julia set of a function. A number of results demonstrating this can be found in ~\cite{ber93}.  The ways in which a particular type of singular value influences the dynamics seems to call for deeper investigations, especially when there are infinitely many singular values. The Baker omitted values are a special type of singular values and some aspects of their importance are already explored in ~\cite{bov}.

\par  A bov is called stable if it is in the Fatou set of the function. This article is devoted mostly to understand the Fatou set of functions with stable bov.

\par  
Though the connectivity of a periodic Fatou component is $1,2$ or $\infty$, that of a pre-periodic Fatou component can be any natural number, as shown in Theorem 6.1, ~\cite{bky-3}.  Wandering domains with connectivity $k$ for every $k>1$ are also known ~\cite{rippon-2008}. None of these works describes the connectivities of all the Fatou components of a given function. This is possible   for meromorphic functions with a stable bov. A Fatou component $V'$ is said to land on a Fatou component $V$ if $V'_{k} =V$ for some $k  \geq 1$. 
The grand orbit of $V$, denoted by $\mathcal{O}(V)$ is the set of all the Fatou components landing on $V_n$ for some $n \geq 0$. Note that every $V' \in \mathcal{O}(V)$ lands on $V$ if $V$ is periodic. For describing our results concerning meromorphic functions with a stable bov, let $$U ~\mbox{ be the Fatou component containing the bov}. $$
 The following is the first result of this article.
 
 \begin{thm}\label{tfc}
 	Let  $f$ be a meromorphic function with a stable bov. Assume that $U$ is the Fatou component of $f$ containing the bov. 
 	\begin{enumerate}
 		\item If $V \notin \mathcal{O}(U)$ and is a multiply connected Fatou component then $V$ lands on a Herman ring  and $1< c(V') < \infty$ for every $V' \in \mathcal{O}(V)$. In other words, if $V \notin \mathcal{O}(U)$ and $V$ is wandering then $c(V')=1$ for all $ V' \in \mathcal{O}(V)$.
 		
 		\item  If $U$ is periodic then  $U'$ is infinitely connected for every $U' \in \mathcal{O}(U)$.
 		
 		\item If $U$ is pre-periodic then each Fatou component landing on $U$ is infinitely connected. Every other Fatou components in $\mathcal{O}(U)$ is either simply connected  or lands on a Herman ring. In the later case,  the connectivity of each of these  Fatou components is bigger than $1$ but  finite.  
 		
 		\item If $U$ is wandering then every Fatou component landing on $U$ is infinitey connected, and all other Fatou components in $\mathcal{O}(U)$ are simply connected. 
 	\end{enumerate}
 	
 \end{thm}

%
%
Theorem~\ref{tfc} can be reworded as: \textit{If $f$ has a stable bov then a
  Fatou component of $f$ is infinitely connected if and only if it lands on $U$. Further, every other Fatou component is either simply connected or lands on a Herman ring}. What calls for further investigation is the possible connectivities of Fatou components landing on Herman rings. We are able to do this under the assumption that the number of critical points whose forward orbits do not meet $U$ is finite. For stating the next result, let  $\mathcal{C}=\mathcal{F}(f) \cap \{c: f'(c)=0~\mbox{and}~f^n(c) \notin U~\mbox{for any}~n\}$.

\begin{thm} \label{finite}
	Let  $f$ be a meromorphic function with a stable bov and $U$ be the Fatou component of $f$ containing the bov. If $\mathcal{C}$, as defined above is a finite set  and $d=\prod_{c_i \in \mathcal{C}}(d_i -1)$ where  $d_i$ denotes the local degree of $f$ at $c_i$ then for $\mathcal{N}=\{1,2,3,\cdots, 2+d\}$, the following are true. 
	
	\begin{enumerate}
		\item If $V \notin \mathcal{O}(U)$ then $c(V') \in \mathcal{N} $ for all $V' \in \mathcal{O}(V)$. 
		\item If $U$ is pre-periodic then  $c(U') \in \mathcal{N}$ for all $U'$ landing on $U_k$ for some $k  \geq 1$. 
		
	\end{enumerate} 
\end{thm}

Clearly, the set $\mathcal{N}$ is independent of all the Fatou components of the function. The assumption that   $\mathcal{C}$ is finite is not unusual. For every non-zero complex number $c$ and natural number $d$,  the point $0$ is the bov of  $\frac{1}{cz^d+e^z}$ and it is the only limit point of its critical values. The details are given in Remark~\ref{onlylimitpoint}. If for some $c$ and $d$, the bov $0$ is in the Fatou set of  $\frac{1}{cz^d+e^z}$ then $\mathcal{C}$ is finite.  
Indeed, this is the case for all meromorphic functions for which the bov is stable, it is the only limit point of its critical values and there are only finitely many critical points corresponding to each critical value.   

\par
It is clear from Theorem~\ref{tfc} that the Julia set of a function with stable bov is always disconnected. It can be totally disconnected only when there is a completely invariant Fatou component. A condition is provided in the next result  ensuring this.
\par \begin{thm}\label{invariant-CIFC}
Let $f$ be a meromorphic function with a stable bov. If an invariant Fatou component of $f$ contains the bov then it is completely invariant. Further, if the invariant Fatou component is an attracting domain and all the critical  values are compactly contained in it then the Julia set is totally disconnected. 
\end{thm}
That the assumptions of Theorem~\ref{invariant-CIFC} can actually be satisfied is demonstrated in Example~\ref{totally} in Section 4, which discusses the dynamics of $f_{\lambda}(z)=\frac{\lambda}{e^z +z}$ for certain real values of $\lambda$. It is shown in Example~\ref{totally}(1) that for $0<\lambda< 0.05$, the Fatou set of $f_{\lambda}$ is a completely invariant attracting domain and the Julia set is totally disconnected. However, the Julia set of a meromorphic function with a stable bov is not always totally disconnected. This is shown in Example~\ref{totally}(2). More precisely,  the Fatou set of $f_{\lambda}, \lambda > e+1$,   is found to contain  a $2$-cycle of attracting or parabolic domains, and the Julia set is disconnected but not totally disconnected. Here both the   attracting (or parabolic) domains are infinitely connected.  Can a meromorphic function with an omitted value have two infinitely connected Fatou components?  This question  was raised in ~\cite{tk4}. To our knowledge,   $f_{\lambda}, \lambda > e+1$, is the first such example answering the question affirmatively.

Example~\ref{totally} starts with $z+e^z$ which has no Baker wandering domain (Remark 3, ~\cite{bov}).
Considering an entire function $g$ having a Baker wandering domain,  we construct meromorphic functions $f_2(z)=\frac{\epsilon}{g(z)}+b$ which have a completely invariant attracting domain containing its bov for suitable values of $\epsilon$ and $b$. The details are discussed in Example~\ref{from-BWD} in Section 4.

\par  Some other possibilities are demonstrated in Example~\ref{two-attracting} and Example~\ref{Siegel} in Section 4.
Example~\ref{two-attracting} provides a function with two invariant attracting domains, one of which contains the bov (hence this is completely invariant). A function having a bov and  with an invariant Siegel disk is given in Example~\ref{Siegel}.
\par 
The next two results deal with Baker domains and Herman rings.

In every cycle of Baker domains, there is always a Baker domain whose boundary contains an asymptotic value  of the function (Theorem 13, ~\cite{ber93}). It is known that if a function has a bov then the bov is the only asymptotic value (Lemma~\ref{bov-implication1}). This restricts the existence of Baker domains. Theorem~\ref{nobakerdomain} proves the non-existence of Baker domains whenever the bov is stable. This theorem also rules out invariant Baker domains even when the bov is not stable.

\begin{thm}\label{nobakerdomain}
Let a meromorphic function have a bov. Then it has no invariant Baker domain. Further, if the bov is stable then the function has no Baker domain of any period.  
\end{thm}

 Invariant Baker domains are  ruled out by Theorem~\ref{nobakerdomain}. If an invariant attracting domain or a parabolic domain contains the bov then it is completely invariant by Theorem~\ref{invariant-CIFC}. Hence  no Herman ring can exist. The next result deals with simultaneous existence of  such $2$-periodic Fatou components (Baker domains, or attracting or parabolic domains containing bov) and a Herman ring. It is important to note that if there is a $2$-periodic Baker domain then the bov cannot be stable. 
  \begin{thm}\label{nohermanring}
 	Let $f$ be a meromorphic function having a bov and $f$ have a $2$-periodic Fatou component which is either
 	\begin{enumerate}
 		\item a Baker domain, or
 		\item an attracting domain or a parabolic domain containing the bov.
 	\end{enumerate} 
 	Then $f$ has no Herman ring.
 \end{thm}
 It is seen in the Example~\ref{totally}(2) that  $\frac{\lambda}{z+e^z}, \lambda> 1+e$ has a $2$-periodic Fatou component containing the bov which is either an attracting domain or a parabolic domain. This function satisfies the hypothesis of this theorem and hence has no Herman ring. It is also relevant to mention here that no Herman ring is known to exist for any meromorphic function with an omitted value.

\par

\par 
Section 2 presents some useful facts that are used in the proofs later. The proofs of all the results are given in Section 3. Some functions satisfying the hypotheses of the theorems stated earlier are discussed in Section 4. The last section mentions few problems arising out of this work.  
\par
 For a Fatou component $V$,  the number of components of $\widehat{\mathbb{C}} \setminus V$ is known as the connectivity of $V$ and is denoted by $c(V)$. We denote a disk with center $z$ and radius $r$ by $D_{r}(z)$ for $z \in \mathbb{C}$ and $r>0$. Throughout the article, $f$ denotes a transcendental meromorphic function with a Baker omitted value.
\section{Some preliminary results}
Some preliminary  results are presented in this section for later use.
The following lemma is essentially due to Bolsch~\cite{bolsch} and is to be repeatedly used. A continuous map $g: \Omega_1 \to \Omega_2$ between two open connected subsets of $\mathbb{C}$ is called proper if the preimage of each compact subset of $\Omega_2$ is compact in $\Omega_1$. Further, if $g$ is analytic then there is a $d$ such that every element of $\Omega_2$ has $d$ preimages counting multiplicity. Here, the multiplicity of a point $z$ is the local degree of $g$ at $z$. This number $d$ is known as the degree of  $g: \Omega_1 \to \Omega_2$.

\begin{lem}\label{RH}
Let $f:\mathbb{C}\rightarrow \widehat{\mathbb{C}}$ be a transcendental meromorphic function. If $V'$ is a  component of the preimage of  an open connected set $V$, then exactly one of the following holds.
\begin{enumerate}
\item There exists a natural number  $d $ such that $f : V' \rightarrow V$ is a proper map of degree $d$. In this case, $c(V')-2=d(c(V)-2)+n$ and $n \leq 2d-2$, where $n$ is the number of critical points of $f$ in $V'$ counting multiplicity. Here, the multiplicity of a critical point is one less than the local degree of $f$ at the critical point. Further, if $c(V) =\infty$ then $c(V') =\infty$. 
\item The function $f : V' \rightarrow V$ is an infinite-to-one map and assumes every value in $V$ except at most two. In this case, $c(V)>2$ implies $c(V')= \infty$.
\end{enumerate}
\end{lem}
We put together parts of Lemma 4 and Lemma 5 of ~\cite{rippon-2008} in the following lemma.
\begin{lem}\label{proper}
Let $V$ be a bounded Fatou component of a  meromorphic function $f$. Then $f: V \to V_1$ is a proper map.
\end{lem}
The following lemma proved in \cite{bov} reveals some   useful properties of functions whenever it has a bov. 

\begin{lem}\label{bov-implication1}
Let   $f$ be a meromrophic function with a bov $b$. Then,
\begin{enumerate}
\item $b$ is an asymptotic value of $f$ and it  is the   only asymptotic value of $f$.
\item  there is an $r_0$ such that for all $0< r<r_0$, 
$f^{-1}(D_r (b))$ is infinitely connected and each  component of $ \mathbb{C}\setminus f^{-1}(D_r(b))$ is bounded.
	
\end{enumerate}
\end{lem}
This lemma leads to some important facts setting a context for the results of this article. The backward orbit $O^{-}(\infty)$ of $\infty$ is the set $\{z \in \widehat{\mathbb{C}}: f^k (z)=\infty~\mbox{for some }~k \geq 0\}$. Note that $\infty$ and all the poles of $f$ are in $O^{-}(\infty)$. By a Julia component, we mean a maximally connected subset of the Julia set.

\begin{lem}\label{singleton}
	Let $f$ be a meromorphic function having a bov. Then $f$ has  infinitely many poles. Further, if the Fatou set contains the bov and $U$ is the Fatou component containing it then, 
	\begin{enumerate}

\item The preimage of $U$ is connected (and hence unbounded) and this is the only unbounded  Fatou component of $f$. Further, it is infinitely connected.
\item If $U'$ is a Fatou component such that $U'_{k} =U$ for some $k > 1$ then $U'$ is infinitely connected.
\item All the components of $\mathcal{J}(f) \cap \mathbb{C}$ are bounded. Further, every Julia component containing a point of $O^{-1}(\infty)$ is singleton.
	\end{enumerate}
\end{lem}

\begin{proof}
	First we make an useful observation that is true even when the bov is not in the Fatou set.
Consider a ball $D_r(b)$ centered at the bov $b$ such that its boundary does not contain any critical value of $f$. Then for sufficiently small $r$, $ f^{-1}(D_r(b))$ is infinitely connected and  each component of $\mathbb{C} \setminus f^{-1}(D_r (b))$ is bounded (by Lemma~\ref{bov-implication1}(2)). The boundary of each such component does not contain any critical point and is a simple closed curve bounding a simply connected domain. Let $D_1, D_2, D_3, \cdots$ be the components of $\mathbb{C} \setminus f^{-1}(\overline{D_r(b))} $. 	Each $D_i$ is mapped onto $\widehat{\mathbb{C}} \setminus \overline{D_r(b)}$	and therefore it contains a pole. Since there are infinitely many $D_i$'s, the function $f$ has infinitely many poles. 
\par
Let the bov $b$ be in the Fatou set and $r>0$ be small enough so that $\overline{D_r(b)} \subset U$.
Then the $D_i$'s, as defined in the previous paragraph have their boundaries in the Fatou set of $f$. In fact, $\mathbb{C} \setminus (\cup_i D_i)$ is contained in the Fatou set.

\begin{enumerate}
		
\item
 The set $f^{-1}(D_{r}(b))$ is connected by Lemma~\ref{bov-implication1}(2). Since each component of $f^{-1}(U)$ intersects $f^{-1}(D_{r}(b))$, the set $f^{-1}(U)$ is connected. Let $U_{-1}=f^{-1}(U)$. It follows again from Lemma~\ref{bov-implication1}(2) that $U_{-1}$ is the only unbounded Fatou component and is infinitely connected. 
\item  If $U'$ is a Fatou component such that $U'_{k} =U$ for some $k > 1$ then  $U'$ is contained in some $D_i$. Hence $U'$ is  bounded. The map $f^{k}: U' \to U_{-1}$ is proper by Lemma~\ref{proper} because the composition of finitely many proper maps is a proper map. Now it follows from Lemma~\ref{RH} that $U'$ is infinitely connected.
\item
Observe that each component of $\mathcal{J}(f) \cap \mathbb{C}$ is contained in some $D_i$ and hence is bounded. It now follows that the Julia component containing $\infty$ is singleton.	A singleton Julia component is mapped onto a singleton Julia component (See Lemma 4, ~\cite{tk-zheng}).
Therefore, every Julia component containing a point of the backward orbit of $\infty$ is singleton. 
	\end{enumerate}
\end{proof}
\begin{rem}\label{disk-simplyconnecteddomain}
 In the proof of the previous theorem, the disk $D_r(b)$ can be replaced by any open connected region containing the bov and compactly contained in $U$.
 
\end{rem}
An asymptotic value of a meromorphic function with finite order is either a limit point of its critical values or all the singularities lying over that asymptotic value are logarithmic (Corollary 1,~\cite{berg-ermk}). In other words, if there is a non-logarithmic singularity lying  over an asymptotic value of a meromorphic function with finite order then this asymptotic value is a limit point of critical values. We prove this statement for all the meromorphic functions with bov regardless of their orders.

\begin{lem}\label{critically infinite}
If a meromorphic function has a bov then the bov   is a limit point of its critical values.
\end{lem}

\begin{proof}
Let $b$ be the bov of a meromorphic function $f$.
Suppose on the contrary that there is a neighbourhood $N$ of $b$  which does not contain any critical value of $f$. Since $b$ is the only asymptotic value of $f$ (by Lemma \ref{bov-implication1}$(1)$),  $f:f^{-1}(N) \rightarrow N\backslash \{b\}$ is a covering. The fundamental group, $\pi_1 (S)$ of a Riemann surface $S$ at a fixed base point $x_0$ is the group of all loops with the same starting and ending point $x_0$ such that none of these is homotopically equivalent to the other.
Note that $N_{-1}=f^{-1}(N)$ is connected in this case. It is well-known that $\pi_1 (N_{-1})$ is a subgroup of $\pi_1 (N\backslash \{b\})$. Since $\pi_1(N\backslash \{b\})$ is isomorphic to the additive group $\mathbb{Z}$, $\pi_1(N_{-1})$ is either trivial or is  isomorphic to $\mathbb{Z}$. But the set $N_{-1}$ is infinitely connected by Lemma \ref{bov-implication1}(2), which means that $\pi_1(N_{-1})$ has infinitely many generators. Hence it is neither isomorphic to the trivial group nor to $\mathbb{Z}$ leading to a contradiction. This proves that the bov is a limit point of critical values.
	
\end{proof}
There are functions for which the bov is the only limit point of critical values.  
\begin{rem}
	For each non-zero complex number $c$ and natural number $d$, the function  $f(z)=c z^d+e^z $ has bov at $\infty$. We prove this by using a known theorem stating that if the image of every unbounded curve under an entire function is unbounded  then $\infty$ is the bov of the function (Theorem 2.2~\cite{bov}). To show that the image of every unbounded curve $\gamma$ under $f(z)=c z^d+e^z $ is unbounded, first note that if for an unbounded sequence $z_n \in \gamma$, the image sequence $\{f(z_n)\}_{n>0}$ is unbounded then $f(\gamma)$ is unbounded. There are three cases depending on the location of the unbounded part of $\gamma$ and we provide such an unbounded sequence on $\gamma$ in each case. 
\begin{itemize}
	\item \textbf{Case-1}: If there is a real number $M$ such that $\gamma_M=\{z \in \gamma: \Re(z)<M\}$ is unbounded then for an ubounded sequence $z_n \in \gamma_M$ (with $\lim_{n \to \infty} |z_n| =\infty$), we have $|f(z_n)| \geq |c| |z_n|^d -|e^{z_n}| \geq |c| |z_n|^d -e^M$. Consequently the image sequence $\{f(z_n)\}_{n>0}$ is unbounded.
	\item \textbf{Case-2}: If there is $\alpha \in (0, \frac{\pi}{2})$ such that $\gamma_{\alpha} =\{z \in \gamma: Arg(z) \in (-\alpha, \alpha)\}$ is unbounded then for an unbounded sequence $z_n =x_n +i y_n \in \gamma_\alpha$ we have $(-\tan \alpha)x_n < y_n < (\tan \alpha)x_n $. Here and later, $Arg(z)$ denotes the principal argument of $z$. Now $|c(z_n)^d|<k  |x_n|^d $ where $k= |c|(1+\tan^2 \alpha)^{\frac{d}{2}}$ and $|f(z_n)| \geq e^{x_n}-k |x_n|^d$. In view of Case(1) above, we assume that $x_n>0$ and $\lim_{n \to \infty} x_n =\infty$. Then $e^{x_n} -k (x_n)^d$, the right hand side of this inequation goes to $\infty$ as $n \to \infty$. This is obvious as $e^{x} -k x^d > (\frac{1}{d !}-k)x^d +\frac{x^{d+1}}{(d+1)!}=x^d (\frac{1}{d!}-k+\frac{x}{(d+1)!})> x^d $ for sufficiently large positive $x$. Therefore $\{f(z_n)\}_{n>0}$ is unbounded.
	
	\item \textbf{Case-3}: If $\gamma$ does not satisfy the conditions of Case(1) and Case(2) then for every $\beta \in (0, \frac{\pi}{2})$, $\gamma_\beta  \cup \overline{\gamma_\beta}$ is unbounded where $\gamma_\beta =\{z \in \gamma: Arg(z) \in (\beta, \frac{\pi}{2})\}$ and $\overline{\gamma_\beta}=\{\overline{z}: z \in \gamma_\beta\}$. There are three situations; either $\gamma_\beta$ or $\overline{\gamma_\beta}$ contains an unbounded curve,  or $\gamma \cap \{z: Arg(z) \in (\beta, \frac{\pi}{2})\}$ (also $\gamma \cap \{z: Arg(z) \in (- \frac{\pi}{2}, -\beta)\}$) contains a sequence of bounded arcs $\gamma_n$ whose one end accumulates at a finite point and whose lengths tend to $\infty$ as $n \to \infty$. The latter follows from the fact that $\gamma$ does not satisfy the condition of Case (2). In the first two cases, the set $\{\Im(z): z \in \gamma_\beta\}$ or $\{\Im(z): z \in \overline{\gamma_\beta}\}$ contains an unbounded interval whereas $\Im(z), z \in \gamma$ takes all possible real values in the third case.

\par
 Either  $\gamma_\beta$ or $\overline{\gamma_\beta}$ is unbounded. Assume that $\gamma_{\beta}$ is unbounded. If $\overline{\gamma_\beta}$ is unbounded the proof goes exactly in the same way. The image of $\gamma_\beta$  under $z \mapsto -cz^d$ (not $z \mapsto cz^d$ !) is contained in a sector with opening $d|\frac{\pi}{2}-\beta|$ (i.e., the principal arguments of the two lines emanating from the origin and  bounding the sector differ by $d|\frac{\pi}{2}-\beta|$). Choose $\beta$ sufficiently near to $\frac{\pi}{2}$ such that $d|\frac{\pi}{2}-\beta| < \frac{\pi}{4}$. Let $S_{\beta}$ be the image of the sector $\{z: Arg(z) \in (\beta, \frac{\pi}{2})\}$ under  $z \mapsto -cz^d$. Then choose $z_0 \in \gamma_\beta$ such that  $-e^{z_0} \in S_{\beta}$, i.e.,  $e^{z_0}$  is in the sector which is the reflection of $S_\beta$ with respect to the origin. This is possible as $\gamma_\beta$ contains points with all possible large positive imaginary part (by the conclusion of the previous paragraph). For the same reason, we can take $z_n \in \gamma_\beta$ such that $\Im(z_n)=\Im(z_0)+2 \pi n$ for all natural numbers $n$.  The points $e^{z_n}$ is on the ray emanating from the origin, containing $e^{z_0}$ and lying in  $-S_\beta =\{-z: z\in S_\beta\}$. Both the sequences $e^{z_n}$ and $-c(z_n)^d$ go to $\infty$ and are in two sectors that are reflections of each other with respect to the origin. Note that $\Re(z_n) \to \infty$ as $n\to \infty$ and therefore $\lim_{n \to \infty}e^{z_n}=\infty$ (as the condition of Case(1) is not satisfied). The number  $|f(z_n)|$ is nothing but the distance between $e^{z_n}$ and $-c(z_n)^d$ and tends to $\infty$ as $n \to \infty$.
\end{itemize}	
	 \label{entire-bov}
\end{rem}
\begin{rem}
 The point at  $\infty$ is the bov for $f(z)=cz^d+e^z$ and consequently, $0$ is the bov of $\frac{1}{f(z)}$.  
The critical points of $\frac{1}{f(z)}$ are the roots of   $c d z^{d-1}+e^{z}$. There are infinitely many roots if $d=1$. For $d>1$, the function  $c d z^{d-1}+e^{z}$ has bov at $\infty$ (by Remark~\ref{entire-bov}) and hence it takes $0$ (in fact every value) infinitely often. This is because every point with only finitely many preimages under an entire function is an asymptotic value of the function whereas an entire function with bov has no finite asymptotic value. Let the critical points of $\frac{1}{f}$  be $z_n, n=1,2,3,\cdots$. Then the critical values are  $\frac{1}{c (z_n)^d-cd (z_n)^{d-1}}$. Since $\lim_{n \to \infty} z_n=\infty$, the critical values tend to $0$ as $n \to \infty$. Therefore, the point $0$ is the bov of  $\frac{1}{f(z)}$ and is the only limit point of its critical values.
	\label{onlylimitpoint}
\end{rem}

\section{Proofs of results}
This section proves all the results of this article.
\subsection{Topology of Fatou components}
Recall that for a Fatou component $V$ of $f$, $V_n$ denotes the Fatou component containing $f^{n}(V)$. We say a Fatou component $V'$ lands on a Fatou component $V$ if $V'_k=V$ for some $k \geq 1$. The following is a restatement of Theorem 3 proved in ~\cite{tk-zheng}. The omitted value mentioned in this lemma is not necessarily a bov. 

\begin{lem}
\label{onefatoucomponent}
Let $f$ be a meromorphic function with at least one omitted value and such that all omitted values be contained in a Fatou component $U$ of $f$. For each Fatou component $V$ with $V_n \neq U$ for any $n \geq 0$ the following are true.
\begin{enumerate}
\item If $U$ is unbounded, then $c(V_n)=1$ for all $n \geq 0$.
\item If $U$ is bounded, then $c(V)=1$ or $V$ lands on a Herman ring.
\item If $U$ is wandering, then $c(U_n)=1$ for all $n \geq 0$.
\item Let $U$ be pre-periodic but not periodic.
If $U$ is unbounded, then $c(U_n)=1$ for all $n \geq 0$. If $U$ is bounded,
then $c(U)=1$ or $U$ lands on a Herman ring.
\item If $U$ is periodic, then $c(U_n)$ is either $1$ or $\infty$ for all $n \geq 0$.
		
	\end{enumerate}
\end{lem}
 Though the above result is silent on the connectivity of the Fatou components landing on $U$, it is the main ingredient in the proof of Theorem~\ref{tfc} which determines the connectivity of every Fatou component whenever the function has a stable bov. We now present a proof of Theorem~\ref{tfc}.

\begin{proof}[Proof of Theorem~\ref{tfc}]   
	
	\begin{enumerate}
\item If $V \notin \mathcal{O}(U)$ then $V_n \neq U$ for any $n$ and it follows from Lemma~\ref{onefatoucomponent}(1-2) that $V$ lands on a Herman ring whenever it is multiply connected. Since for each $V' \in \mathcal{O}(V)$, $V'$ is bounded, there is a $k$ such that $V'_k$ is a Herman ring and $f^{k}: V' \to V'_{k}$ is proper by Lemma~\ref{proper}. Further, $c(V')=2+N$ by Lemma~\ref{RH}, where $N$ is the number of critical points of $f^k$ in $V'$ counting multiplicity. In other words, $1< c(V')  < \infty$.
 
\item 
Let $U$ be periodic and $U' \in \mathcal{O}(U)$. Then, the preimage $U_{-1}$ of $U$ is periodic and there is a $k \geq 0$  such that $f^k (U') \subseteq U_{-1}$. The map $f^k: U' \to U_{-1}$ is either proper  or is an infinite-to-one map. Since $U_{-1}$ is infinitely connected, $U'$ is infinitely connected by Lemma~\ref{RH} in both the cases.

\item 
If $U$ is pre-periodic then each Fatou component landing on $U$ is infinitely connected by Lemma~\ref{singleton}(2).

If a Fatou component $U'$ in $\mathcal{O}(U)$ not landing on $U$, is multiply connected then it lands on a Herman ring $H$ by Lemma~\ref{onefatoucomponent}(2). 
The map $f^k: U' \to H$ is proper for some $k$.  Now it follows from Lemma~\ref{RH} that  $1 < c(U') < \infty$. 

\item If $U$ is wandering then every Fatou component landing on $U$ is infinitely connected by Lemma~\ref{singleton}(2). Let $U'$ be a Fatou component in $\mathcal{O}(U)$ not landing on $U$. If 
$U'$ is multiply connected then it lands on a Herman ring by Lemma~\ref{onefatoucomponent}(2), giving that $U'$ is periodic or pre-periodic, which is not true. Thus $U'$ is simply connected for all $U' \in \mathcal{O}(U)$. 
\end{enumerate}
	
\end{proof}

The proof of Theorem~\ref{finite} follows.

\begin{proof}[Proof of Theorem~\ref{finite}]
Recall that $\mathcal{C}=\mathcal{F}(f) \cap \{c: f'(c)=0~\mbox{and}~f^n(c) \notin U~\mbox{for any}~n\}$. This set is finite, by assumption. Let  $\mathcal{C} = \{c_1, c_2,c_3, \cdots c_l\}$ and $d= \prod_{i=1}^l (d_i-1)$ where $d_i$ is the local degree of $f$ at $c_i$.  The desired finite set of connectivities is going to be $\mathcal{N}=\{1,2,3,\cdots, 2+d\}$. 
\begin{enumerate}

\item Let $V \notin \mathcal{O}(U)$. If $V$ is multiply connected then all the Fatou components in $\mathcal{O}(V)$ are multiply connected, by Theorem~\ref{tfc}. Further, each Fatou component in $\mathcal{O}(V)$ lands on a Herman ring. Clearly, this Herman ring is the same, say $H$ for all Fatou components in $\mathcal{O}(V)$.  Let $V' \in \mathcal{O}(V)$. Then $V'_{k} =H$ for some $k$. Since $H$ and its forward iterated images cannot contain any critical point, $V'_{n} $ intersects $\mathcal{C}$ for at most finitely many values of $n$ and a single $V'_{n}$ cannot contain any critical point more than once. In other words, $f^k: V' \to H$ is a proper map with degree at most $d$ and, by Lemma~\ref{RH}, $c(V')=2+ j$ where $j \leq d$. In other words, $c(V') \in \{2, 3, 4, \cdots, 2+d \}$. 
\par If $V$ is simply connected then $V'$ is simply connected for all $V' \in \mathcal{O}(V)$. 
\par Thus $c(V') \in \mathcal{N} $ for all $V' \in \mathcal{O}(V)$.
\item  Let $U' $ be a Fatou component landing on $U_k$ for some $k \geq 1$. Then it follows from  Theorem~\ref{tfc}(3) that either $c(U')=1$ or $U'$ lands on a Herman ring. Following the same argument as in (1), it is concluded that $c(U') \in \mathcal{N}$.
	\end{enumerate}
\end{proof}
\begin{rem}
Let $\tilde{\mathcal{O}}(U) $ be the set of all elements of $\mathcal{O}(U)$ which do not land on $U$. Then it follows from Theorem~\ref{tfc}(3) that if one Fatou component in $\tilde{\mathcal{O}}(U)$ is multiply connected then so are all others.
\end{rem}
 For proving Theorem~\ref{invariant-CIFC}, we need the following lemma, which is slightly more general than what is required. 
\begin{lem}\label{preimageconn}
Let $U$ be an invariant Fatou component of a meromorphic function $f$. If $f^{-1}(D)$ is connected, for some disk $ D \subset U$ then $U$ is completely invariant. In particular, every invariant Fatou component $U$ is completely invariant if it contains an omitted value over which there is only one singularity.
\end{lem}
\begin{proof}
Since $f^{-1}(D)$ is connected, $f^{-1}(U)$ is connected by Theorem 1 of ~\cite{herring}. The forward invariance of $U$ gives that $U \cap f^{-1}(U)$ is non-empty. As $f^{-1}(U)$ is contained in a Fatou component, $f^{-1}(U)\subset U$. In other words, $U$ is backward invariant. Therefore $U$ is completely invariant. \\
If $U$ contains an omitted value over which there is only one singularity then $f^{-1}(D)$ is connected, for each sufficiently small ball $D$ contained in $U$ and  centered at the omitted value. Therefore,  every invariant Fatou component $U$ is completely invariant if it contains an omitted value over which there is only one singularity.
\end{proof}

 \begin{proof}[Proof of Theorem~\ref{invariant-CIFC}]
 	Let $f$ be a meromorphic function with a stable bov.
As evident from Lemma~\ref{bov-implication1}, there is only one singularity of  $f^{-1}$ lying over a bov.
This along with Lemma~\ref{preimageconn} proves that the invariant Fatou component $U$ containing the bov is completely invariant. 

\par 

Consider the closure $K $ of a simply connected subdomain of $ U$  such that $K$ contains the attracting fixed point, the bov and all the singular values of $f$.
Let $D$ be the closure of $\bigcup_{n \geq 0} f^{n}(K)$. Then $D$ is a forward invariant closed and connected subset of $U$. 
Now each Julia component other than the one containing  $\infty$ is contained in a component of $\mathbb{C} \setminus f^{-1}(D) $ (See Remark~\ref{disk-simplyconnecteddomain}).
 Let $J$ be a Julia component of $f$ and $z \in J$. If $\lim_{k \to \infty}f^{n_k}(z) =z^* \in \mathbb{C}$ for some subsequence $n_k$ and  the Julia component containing $z^*$ is $J^*$ then $J^* \subset V$ for some component $V$ of  $\mathbb{C} \setminus f^{-1}(D) $. Since $f^{-1}(D)$ is connected, $V$ is simply connected.  Further, there is a $k_0$ such that $f^{n_k}(J) \subset V$ for all $k>k_0$. Since all the singular values and their forward orbits are in $D$, $V$ does not intersect the forward orbit of any singular value, and hence each branch $f_k$ of $f^{-n_k}$ is a well-defined analytic map in $V$ for every $k > k_0$. Note that each $f_{k}(V)$ contains $J$. It is known that $\{f_k\}_{k>0}$ is normal on $V$ (for example, see Proposition 2~\cite{zheng-singu-2003}). 
Let $f^*$ be a limit function of a convergent subsequence of $\{f_k\}_{k>0}$ on $V$. Then $f^*$ is analytic by the Weirstrass Theorem. If $f^*$ is non-constant then it is observed that  $f^{*}(V)$ is an open set intersecting $J$.  Now a disk in $f^*(V)$ can be found which intersects the Julia set. The sequence $\{f^{n_k}\}_{k>0}$ is normal in this disk. However this cannot be true proving that $f^{*}$ is constant. Since $J$ is contained in the closure of $f^*(V)$, $J$ is singleton. Using a conformal conjugate of $f$, it can be shown that $J$ is singleton even if, for a $z \in J$ and a subsequence $n_k$, $\lim_{k \to \infty} f^{n_k}(z) = \infty$.  Here, it  is important to note that all the poles are simple and $\infty$ is not a critical value, by the hypothesis of this theorem.
Thus, the Julia set of $f$ is totally disconnected.

\end{proof}

\subsection{Baker domains and Herman rings}
We now provide the proof of Theorem~\ref{nobakerdomain}.

\begin{proof}[Proof of Theorem~\ref{nobakerdomain}]
	First we prove that  if the bov is stable then there is no Baker domain of any period.
Let  $U$ be the Fatou component containing the bov $b$.  Let $B$ be a Baker domain of $f$ and $\{B=B_0, B_1,\dots,B_{p-1}\}$ be the corresponding cycle. Suppose that $f^{np} \rightarrow l_i$ locally uniformly on $ B_i$ for $i=0, 1, 2,\dots, p-1$. Then there exists $j$ such that  $l_j =\infty$ and $B_j$ is unbounded. Without loss of generality assume that $j=0$. But it follows from Lemma \ref{singleton}$(2)$ that the set $U_{-1}=\{z:f(z)\in U\}$ is the unique unbounded Fatou component of $f$ and so $U_{-1}=B$. Take $z_0\in B$. Then $z_1=f^p(z_0)\in B$. For a simple curve $\gamma$ in $B$ joining $z_0$ and $z_1$ and containing these points, $f^p(\gamma)$ is a curve in $B$ joining $z_1$ and $z_2=f^p(z_1)$. Let $\gamma_0=
\gamma \cup f^p(\gamma)\cup f^{2p}(\gamma)\dots,$ and $\gamma_i=f^i(\gamma_0)$ for $i \in \{1, 2, 3,\dots, p-1\}$. 
We claim that $f^p(z) \to \infty$ as $z \to \infty$ along $\gamma_0$. The claim will be proved by showing that $\lim_{k \to \infty} f^p (w_k) = \infty$ whenever $w_k$ is a sequence on $\gamma_0$ converging to $\infty$. Let $\Omega$ be a  ball around $\infty$ with respect to the spherical metric. Since $\gamma$ is a compact subset of $B$ and $\lim_{n \to \infty} f^{np}(z)= \infty$ uniformly on $\gamma$, there exists a natural number $n_0$ such that 
\begin{equation}
f^{np}(\gamma) \in \Omega~\mbox{for all}~ n > n_0.	\label{misc1}\end{equation}  Each iterated image of $\gamma$ under $f^p$ is bounded and that is why  contains at most finitely many $w_k$s. More precisely, there is a $k_1$ such that $w_k \notin \cup_{n=1}^{n_0 } f^{np}(\gamma)$ for any $k >k_1$. In other words, for each $k >k_1$,  $w_k \in f^{np} (\gamma)$  for some $n >n_0$ giving that $f^p(w_k) \in f^{(n+1)p}(\gamma)$. Now it follows from Equation~\ref{misc1} that  $f^p(w_k) \in \Omega$ for all $k > k_1$. 
	
\par It follows from the previous paragraph that,
\begin{equation}\label{baker}
\mbox{ each}~  B_i ~\mbox{ contains}~  \gamma_i ,  f^p(\gamma_i)\subset \gamma_i  ~\mbox{and}~  f^p(z)\rightarrow l_i  ~\mbox{as}~  z\rightarrow l_i ~\mbox{along}~  \gamma_i.
\end{equation} 

In particular, $f(z)\rightarrow l_1$ as $z\rightarrow \infty$ along $\gamma_0$, where $l_1\in \partial B_1$. But $\gamma_0$ is an unbounded curve. Its image under $f$ accumulates at the bov which gives that $l_1=b$. This follows from the definition of bov (See Theorem 2.2,~\cite{bov}). However, this is not possible as $b \in B_1$. Therefore, the Fatou set of $f$ cannot contain any Baker domain whenever the bov is stable.	
\par 
By definition, every invariant Baker domain of a meromorphic function $f$ is unbounded and $\infty$ is the limit point of $\{f^n\}_{n>0}$ on it. By Equation~\ref{baker}, $\infty$ is an asymptotic value of $f$. But $f$ has no asymptotic value other than the bov by Lemma \ref{bov-implication1}$(1)$. This proves that $f$ has no invariant Baker domain.
	
\end{proof}

 For proving Theorem~\ref{nohermanring}, we need a result, which is  Theorem 3.8 of ~\cite{relevantpole}. A definition is required to state this.

\par

Let $H$ be a $p$-periodic Herman ring of a function $f$ with an omitted value and $\{H_0,~H_1,\dots ,H_{p-1}\}$ be the cycle of $H$, where $H=H_0=H_p$.  
Given a Herman ring $H$, a pole $w$ of $f$ is said to be $H$-relevant if some ring $H_i$ of the cycle containing $H$ surrounds $w$ i.e., the bounded component of $\widehat{\mathbb{C}} \setminus H_i$ contains $w$. 

\begin{lem}\label{periodicfatoucomponent}
	Let  $V$ be a periodic Fatou component of a meromorphic function $f$ such that its closure contains at least one omitted value. Then for every Herman ring $H$ of $f$, the number of $H$-relevant poles is strictly less than the period of $V$. In particular, if $V$ is invariant or $2$-periodic then $f$ has no Herman ring.
	
\end{lem}

A proof of Theorem~\ref{nohermanring} follows.
\begin{proof}[Proof of Theorem~\ref{nohermanring}]
Let $\{U_1,~U_2 \}$ be a cycle of $2$-periodic Fatou components of $f$. If $U_1$ is a Baker domain  then $U_i$ is unbounded and $f^{2n} \to \infty$ locally uniformly on $U_i$ for some $i$. Without loss of generality assume $i=1$. Since the bov is the only asymptotic value, it follows from the proof of the previous theorem that the bov is on the boundary of $U_2$. Thus the bov is in the closure of $U_2$. If $U_1$ is an attracting domain or a parabolic domain containing the bov then clearly its closure contains the bov. Now, it follows from Lemma~\ref{periodicfatoucomponent} that $f$ has no Herman ring.
\end{proof}

\section{Examples } 
 A Baker wandering domain of an entire or a meromorphic function $f$ is a multiply connected wandering domain $W$ such that $f^n \to \infty$ locally uniformly on $W$ such that $c(W_n) >1$ for all $n$ and $W_n$ surrounds the origin for all  sufficiently large $n$.  These domains are not only bounded but also ensures that all other Fatou components are bounded.   
If an entire function has a Baker wandering domain then $\infty$ is its bov (Theorem $2.3$, \cite{bov} ). That the converse is not true is demonstrated by the function $\lambda + z +e^z$ for $0 \leq \lambda < 1$. One can show that there are infinitely many unbounded invariant attracting domains  for this function ruling out Baker wandering domains even though $\infty$ is its bov~(Remark 3, \cite{bov}). The function $\frac{1}{e^z +z}$ and some of its variants are used to construct examples demonstrating some of the results proved in the previous sections.

\subsection{ The function $ \frac{\lambda}{e^z +z}$}
Let $ f(z)=\frac{1}{e^z +z}$ and $ f_{\lambda}(z)=\lambda f(z), \lambda >0$.
 Following are few basic properties of $f$ on the real line, the proofs of which follow by analysing the first and second derivatives of the function. 
\begin{obs}\label{bpf}

\begin{enumerate}
\item The function $f$ is differentiable on $ \mathbb{R}$ except at the pole  $x_0\approx -0.55$.
\item Since $f^\prime(x)= -\frac{1+e^x}{(x+e^x)^2} < 0$ for all $x \in \mathbb{R} \setminus \{x_0\}$, the function  $f$ is decreasing in $(-\infty, x_0)$ and in $(x_0, \infty)$.
\item The function $f^{\prime}$ is increasing in $[0, \infty)$ because $f^{\prime \prime}(x) = \frac{2+4e^x+e^x(e^x-x)}{(x+e^x)^3}>0$ for all $x \geq 0$. This follows by observing that $e^x -x >0$ for all $x \geq 0$.
\item The function  $f^2$ is increasing in $(-\infty, x_{-1}), (x_{-1}, x_0)$ and $(x_0, \infty)$ where $x_{-1}$ is the unique real preimage of $x_0$. This is so because  $(f^2)^\prime(x)=f^\prime(f(x))f^\prime(x) >0 $ for all $x \in \mathbb{R}\setminus \{ x_{-1}, x_0\}$.
\item  $\lim\limits_{x\rightarrow x_0^+}f(x)=\infty$ and $\lim\limits_{x\rightarrow x_0^-}f(x)=-\infty$.

\end{enumerate}
\end{obs}
 
The graphs of $f$ and $f^2$ are given in Figure~\ref{graph1} and computationally it is found that there is a positive fixed point $p\approx0.49$ and a negative fixed point $q \approx -1.16$ of $f$.

\begin{figure}[H]
	\centering
\subfloat[ Graph of $f(x)$]
{\includegraphics[width=1.73in,height=1.73in]{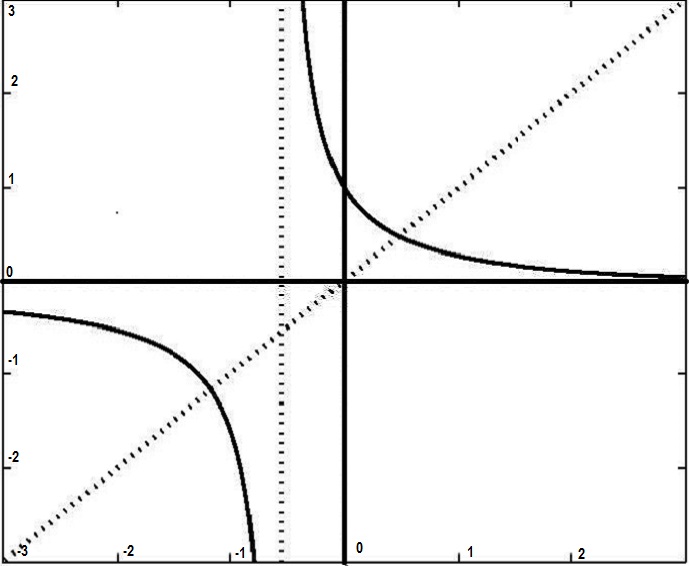}}
\hspace{0.00001cm}
\subfloat[Graph of $f^2(x)$]
{\includegraphics[width=1.73in,height=1.73in]{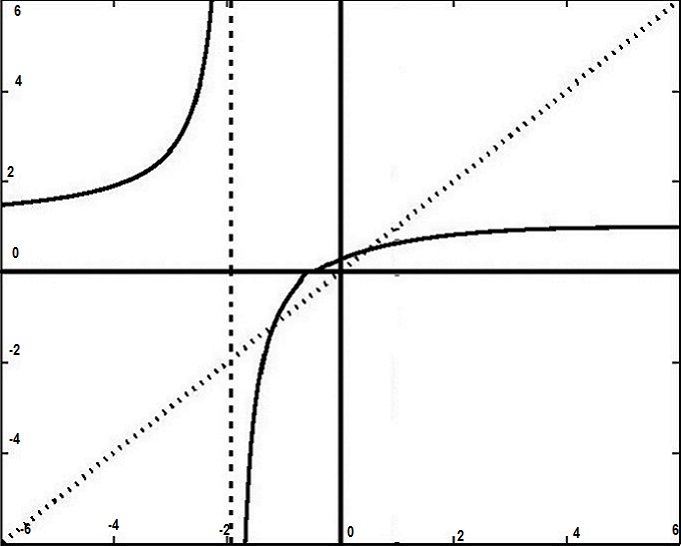}}
\hspace{0.00001cm}
\subfloat[Graph of $\phi(x)= \frac{e^x(1-x)}{(e^x+x)^2}$]
 {\includegraphics[width=1.73in,height=1.73in]{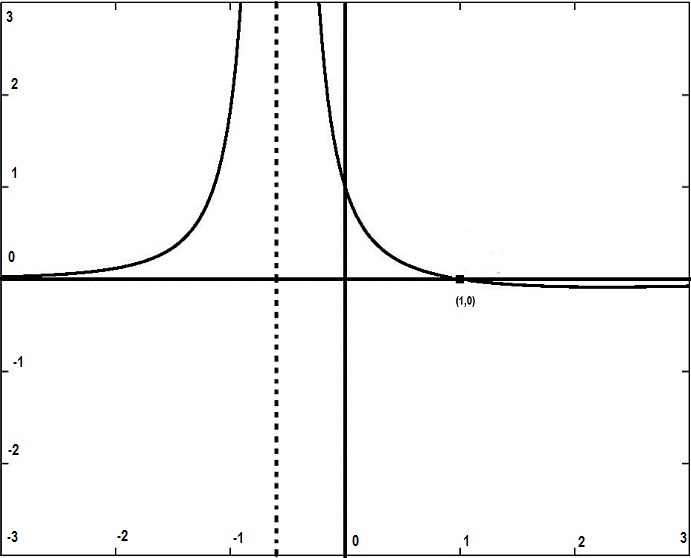}}
\caption{Graphs }	
\label{graph1} \end{figure}

For each $k\in \mathbb{Z}$,  $z_k=i\pi(2k+1) $ is a critical point of $f$.  There is exactly one critical point corresponding to each critical value. In other words, exactly one preimage of each critical value is a critical point.  
Since $\infty$ is the bov of $z + e^z$ (See Remark~\ref{onlylimitpoint}), $0$ is the bov  of $\frac{1}{e^z+z}$. There are no other asymptotic values of $f$ by Lemma~\ref{bov-implication1}. 
\par  The point $0$ is the bov  of $f_{\lambda}$ for every $\lambda \in \mathbb{C} \setminus \{0\}$. The critical values of $f_\lambda$ are of the form  $ \frac{\lambda}{-1+i\pi(2k+1)}$  for $ k\in \mathbb{Z}$ and its limit point is $0$. Here the bov is the only limit point of all singular values.
Note that $x_0$ remains to be a pole of $f_\lambda$.

%

\begin{ex}\label{totally}
Let $f_\lambda(z)=\frac{\lambda}{e^z+z}$ for $\lambda >0$.
\begin{enumerate}
\item (\textbf{Totally disconnected Julia sets}) For $0<\lambda <0.05$, the Fatou set  of $f_\lambda$ is a completely invariant attracting domain containing the bov and the Julia set is totally disconnected.
 
\item  (\textbf{Non-invariant multiply connected Fatou components}) For { $\lambda >1+e$}, the Fatou set of $f_\lambda$ contains a $2$-cycle of attracting or parabolic domains containing the bov and the Julia set is disconnected but not totally disconnected.
\end{enumerate}

\end{ex}

\begin{proof}

 Let $g_{\lambda}(x)=\lambda f(x)-x $ for $x\in \mathbb{R}$ and $\lambda >0$. The roots of $g_{\lambda}(x) $ are the fixed points of $ f_{\lambda}$. Since $f_{\lambda}^{\prime}(x)= \lambda f^\prime(x)<0$, $g_{\lambda}^{\prime}(x)=\lambda f^{\prime}(x)-1<0$ for  $ x\in \mathbb{R} \setminus \{x_0\} $ giving that $g_\lambda$ is strictly decreasing. Now $g_{\lambda}(0)=\lambda>0$, $\lim\limits_{x\rightarrow +\infty}g_{\lambda}(x)=-\infty$ and $g_{\lambda}(x)$ is continuous on $[0, \infty]$. By the Intermediate Value Theorem, there is a unique positive real number $x_{\lambda}$ such that $g_{\lambda}(x_{\lambda})=0$. 
  Since $f_\lambda(x)>0$ for all $x\in (x_0,0)$, $f_\lambda$ has no fixed point in $(x_0,0)$ and hence, $x_\lambda$ is the only fixed point of $f_\lambda$ in $(x_0, \infty)$.
 
 Now $\lim\limits_{x\rightarrow -\infty}g_{\lambda}(x)=\infty$ and $\lim\limits_{x\rightarrow x_0^-}g_{\lambda}(x)=-\infty$. Again by the Intermediate Value Theorem, there is a negative real number 
$\tilde{x_{\lambda}}$ such that $g_{\lambda}(\tilde{x_{\lambda}})=0$.
Therefore, the only real fixed points of $f_\lambda$ are $x_\lambda$ and   $\tilde{x_{\lambda}}$.
\par
 In order to determine the nature of these fixed points we define
$\phi(x)=xf^{\prime}(x)+f(x)=\frac{- x(e^{x}+1)}{(e^{x}+x)^2}+\frac{1}{e^{x}+x}=\frac{-x(e^{x}+1)+e^{x}+x}{(e^{x}+x)^2}=\frac{e^x(1-x)}{(e^{x}+x)^2} ~\mbox{for}~ x \geq 0$.
Observe that 
\begin{gather}{\label{200}} 
 \phi(x)
  \begin{cases}
  >0,~for~x<x_0\\
  >0,~for~x_0<x<0\\
   >0,~for ~0\leq x< 1\\
    =0,~for~x=1\\
     <0,~ for~ x>1
  \end{cases}.
\end{gather}
 Further, $ \phi(x)\rightarrow 0$ as $x\rightarrow \infty$. See Figure~\ref{graph1}.

\begin{enumerate}
	\item 

 Note that $g_\lambda(0)=\lambda>0$. For $0<\lambda<0.05$, $g_\lambda(1)=\frac{\lambda}{e+1}-1<0$. Thus $0<x_\lambda<1$. It follows from Equation  \ref{200} that $\phi(x_\lambda) > 0$, i.e., $x_\lambda f^{\prime}(x_\lambda)+f(x_\lambda) >0$. This gives that
$ \lambda f^{\prime}(x_\lambda)+\lambda\frac{f(x_\lambda)}{x_\lambda}>0$ and consequently, $\lambda  f^{\prime}(x_\lambda) >-1$. Since $\lambda  f^{\prime}(x)<0$ for $x >0$,  $x_\lambda$ is an attracting fixed point of $f_\lambda$. 

\par 
 It follows from Equation \ref{200} that $\phi(x) >0$ for all $x<x_0$ which means that $\phi(\tilde{x_{\lambda}}) > 0$. This gives that $\tilde{x_{\lambda}}f^\prime(\tilde{x_{\lambda}})+f(\tilde{x_{\lambda}})>0$, and $\lambda\tilde{x_{\lambda}}f^\prime(\tilde{x_{\lambda}})+\lambda f(\tilde{x_{\lambda}})>0$ as $\lambda >0$. Further, since $\tilde{x_\lambda} <0$, $f_\lambda^\prime(\tilde{x_{\lambda}})<-1$. Therefore $\tilde{x_{\lambda}}$ is repelling.

 For $0<\lambda<0.05$, $f^\prime_\lambda(0)=-2\lambda>-1$. Since $f^\prime_\lambda$ is increasing in $[0,\infty)$ by Observation $\ref{bpf}(3)$, $-1 < f^\prime_\lambda(x) < 0$ for all $x \geq 0$. For $h_\lambda(x) = f^2_\lambda(x)-x$, $0<\lambda< 0.05$, observe that $h^\prime_\lambda(x)<0$ in $[0,x_\lambda)$. Note that $h_\lambda(x_\lambda)=0$ and $h_\lambda$ is decreasing in $[0, x_\lambda]$. Hence, $h_\lambda(x)>0$ for all $x \in [0,x_\lambda)$. As $f_\lambda^2$ is increasing, $f^2_\lambda(x)>x$  and $f^{2n}_\lambda(x)< x_\lambda $ for $0<x<x_\lambda$. Therefore, $\{f^{2n}_\lambda(x)\}_{n>1}$ is convergent for all $0 \leq x \leq  x_\lambda$. This limit point is nothing but $x_\lambda$. Since $f_\lambda([0,x_\lambda]) \subsetneq  [x_\lambda, \infty)$ and $f_{\lambda}([x_\lambda, \infty)) \subsetneq (0, x_\lambda]$, $f^{2n+1}_\lambda(x) \rightarrow f_\lambda(x_\lambda)=x_\lambda$ for all $x \in [0, x_\lambda]$ and consequently,  $\lim_{ n \to \infty}f^n_\lambda(x) =x_\lambda$ for all $x \in [0, \infty)$. The invariant attracting domain $U$ corresponding to $x_\lambda$ contains $0$, the bov. Therefore,  $U$ is completely invariant by Theorem \ref{invariant-CIFC}.
\par 
 For  $|z|=0.5$, $|\frac{\lambda}{z+e^z}|\leq \frac{\lambda}{|e^z|-|z|}= \frac{\lambda}{|e^z|-0.5}= \frac{\lambda}{e^{\Re(z)}-0.5}$. Note that the function $x \mapsto e^x -0.5$ is a strictly increasing function in $[-0.5,0.5]$ attaining its minimum at $-0.5$. The minimum value is approximately equal to $0.106$.  Hence 
 
 $|\frac{\lambda}{z+e^z}| \leq \frac{\lambda}{0.1} <0.5 $. In other words, $f_\lambda (D)$ is strictly contained in $D$ where $D= \{z: |z| \leq 0.5\}$, and consequently $\{f_\lambda ^n\}_{n>0}$ is normal on $D$ by the Fundamental Normality Test. The Fatou set of $f_\lambda$ contains $D$. Now observe  that all the critical values of $f_\lambda$ are of the form $\frac{\lambda}{-1+i\pi(2k+1)}$ for $ k\in \mathbb{Z}$ and the maximum of their moduli is $\frac{\lambda}{\sqrt{1+\pi^2}} < 0.5$ (which is attained for $k=0$).  Thus $D$ contains all the singular values of $f_\lambda$ and their forward orbits. Since $U$ contains at least one singular value of $f_\lambda$, $U \cap D \neq \emptyset$ and hence $D \subset U$. It follows from Theorem~\ref{invariant-CIFC} that the Julia set of $f_\lambda$ is totally disconnected for $0< \lambda< 0.05$.


\item  For $\lambda >1+e$,  $g_\lambda(1) = \frac{\lambda}{1+e}-1 >0$ and $\lim\limits_{x\rightarrow +\infty}g_{\lambda}(x)=-\infty$ implying that $x_\lambda >1$. By Equation \ref{200}, $\phi(x_\lambda)<0.$ It gives that $\frac{\phi(x_\lambda)}{f(x_\lambda)}=f_{\lambda}^{\prime}(x_\lambda)+1<0$ and hence $f_{\lambda}^{\prime}(x_\lambda)<-1$.   
Therefore, $x_\lambda$ is a repelling fixed point of $f_\lambda$ for all $\lambda>1+e$.

\par  Let $h_\lambda(x) = f^2_\lambda(x)-x$. Note that $h_\lambda(0)= f^2_\lambda(0) > 0$ and $h_\lambda(1)=  \frac{\lambda}{  \frac{\lambda}{1+e}+ e^{\frac{\lambda}{1+e}}}-1$. Taking $\alpha=\frac{\lambda}{1+e}$, it is observed that $ \psi(\alpha)=\frac{\alpha (1+e)}{\alpha+e^{\alpha}} 
 < 1$ for all $\alpha >1$. In other words, $h_\lambda(1)< 0$. Thus there exists  $a_1 \in (0, 1)$ such that $h_\lambda(a_1)=0$ and hence $f^2_\lambda(a_1)=a_1$.  Choose $a_1$ to be the smallest positive $2$-periodic point. Here we do not rule out the possibility of other $2$-periodic points. Thus
 \begin{gather}{\label{f2}} 
f^{2}_\lambda(x)-x
\begin{cases}
 >0~for ~0\leq x< a_1\\
 =0~for~x=a_1 
\end{cases}.
\end{gather}

Since $f^2_\lambda$ is an increasing function on $[0,a_1]$, using Equation \ref{f2} we have $x < f^2_\lambda(x) < f^4_\lambda(x)< \dots < a_1$, for all $x \in [0, a_1)$ and $\lim\limits_{n\rightarrow \infty}f_\lambda^{2n}(x)=a_1$. This implies that $a_1$ is either an attracting or a parabolic periodic point of period $2$.

\par 
Let $f_\lambda(a_1)=a_2$ and $U_i$ be the attracting or the parabolic domain containing  $a_i$ for $i=1,2$. It is already seen that $[0, a_1) \subset U_1$. Since  $f_\lambda$ maps $(x_0, a_1)$ onto $(a_2, \infty)$ and  $(a_2, \infty)$ onto $(0, a_1)$, $U_2$ is unbounded.
  It is infinitely connected by Theorem~\ref{tfc} (2). It also follows from Theorem \ref{tfc} that $U_1$ is bounded and infinitely connected. Therefore, the Julia set is disconnected but not totally disconnected.
\end{enumerate}
\end{proof}

For $\lambda > e+1$, the above theorem  does not describe the Fatou set completely. The following remarks contain  some more information on the dynamics of $f_\lambda$. 
\begin{rem}

\begin{enumerate}

\item It is worth mentioning that $x_\lambda \rightarrow 0$ whenever $\lambda \rightarrow 0$. In order to verify this, note that $\lambda = x_\lambda(x_\lambda +e^{x_\lambda})$ and the function $x \mapsto x(x+e^x)$ is strictly increasing. In particular, it is a homeomorphism of $(0,\infty)$ and $x_\lambda$ is a continuous function of $\lambda$. 

\item The existence of $2-$periodic points are proved in the above example. We claim that $f_\lambda $ does not have any real positive periodic point of period greater than two. Note that $f_\lambda((x_0,x_\lambda])=[x_\lambda, \infty)$ and $f_\lambda([x_\lambda, \infty))=(0, x_\lambda] $. This shows that $f_\lambda$ does not have any positive periodic point of odd period. 
If  $f_\lambda$ has a periodic point $r$ of even period greater than two then either $f_\lambda^2(r)>r$ or $f_\lambda^2(r)<r$. Since $f^2_\lambda$ is increasing, either $r < f^2_\lambda(r)< f^4_\lambda(r) < ...$ or $r > f^2_\lambda(r)> f^4_\lambda(r)>...$ proving that $f$ cannot have any periodic point of even period bigger than $2$. 
\item It remains to be checked whether $f_\lambda$ has a $2$-periodic point different from the cycle $\{a_1, f_{\lambda}(a_1) \}$.

\end{enumerate}

\end{rem}
 
\subsection{Attracting domains from Baker wandering domains}
Consider the function $Cz^2 \prod_{n=1}^{\infty} (1+\frac{z}{\gamma_n})$ where $1<\gamma_1<\gamma_2<....<\gamma_n<..$, $C= \frac{1}{4e}$, $\gamma_1>4e$ and $\gamma_{n+1}=C\gamma_n^2(1+\frac{\gamma_n}{\gamma_1})(1+\frac{\gamma_n}{\gamma_2})\dots (1+\frac{\gamma_n}{\gamma_n})$.  Baker proved that this function has Baker wandering domains~\cite{bwd}. The point at $\infty$ is the bov of this entire function. In fact, this is true for all entire functions with a Baker wandering domain, as observed in ~\cite{bov}. This fact is crucial for the next example.

\begin{ex}
Let $g$ be an entire function having a Baker wandering domain $W$ and $b\in W \setminus g^{-1}(0)$. Then $b$ is the bov of $f_2(z)=\frac{\epsilon}{g(z)}+b$ for every $\epsilon >0$. The bov is contained in a completely invariant attracting domain of $f_2$ for suitable values of $\epsilon$.
\label{from-BWD}
\end{ex}

\begin{proof}
Since $g$ has a Baker wandering domain, $g$ does not omit any point in $\mathbb{C}$. In particular, the point $0$ is not an omitted value of $g$.
 Since $b\in W \setminus g^{-1}(0)$, there exists a disk $D_{r}(b)$ inside $W$ such that  $0 \notin \overline{g(D_r(b))}$. Choose $\epsilon >0$ such that $|\frac{\epsilon}{g(z)}|<\frac{r}{2}$ for all $z \in D_r (b)$. This is possible because $g(D_r (b))$ is a bounded set away from the origin. Then $f_2 (D_r (b) ) \subsetneq  D_r(b)$ and consequently  $f_2^n(D_r(b))\subsetneq D_r (b) $ for all $n$. This gives that $D_r (b)$ is contained in a Fatou component of $f_2$, by the Fundamental  Normality Test. This Fatou component $U$
  is an attracting domain. Indeed, the ball $D_r (b)$ contains a fixed point and that must be attracting by the Schwarz's lemma. Further, $U$ is invariant and contains $b$, the bov of $f_2$.  It follows from Theorem \ref{invariant-CIFC} that $U$ is completely invariant.  
\end{proof}
 \subsection{Two invariant attracting domains}
\begin{ex} There is a meromorphic function with stable bov such that it has two invariant attracting domains. One of these attracting domains contains the bov and is completely invariant.
	\label{two-attracting}
\end{ex}
 \begin{proof}
Consider $f_3(z)=\frac{0.1}{z^9 +e^z}-0.99$. There is a unique real pole at $x_3 \approx -0.904$ and the function $g_{3}(x)=f_{3}(x)-x$ is strictly decreasing in $(-\infty, x_3) \cup (x_3, \infty)$. 
\par 
 Since $\lim_{x \to x_{3}^+} g_{3}(x)=\infty$ and $g_{3}(-0.72) \approx -0.04 <0$,  $f_3$ has a fixed point $a_1$ in $I_1 =[x_3,- 0.72]$. Similarly, $g_{3}(-1.069) \approx 0.011>0$ and $g_{3}(-1) \approx -0.148<0$ gives that  $f_3$ has a fixed point $a_2$ in   $ I_2 = [- 1.069, -1]$. Note that each fixed point $\hat{z}$ of $f_{3}$ satisfies $\frac{0.1}{z^9 +e^z}=z+0.99$ and its multiplier is $-10(\hat{z}+0.99)^2 (9 \hat{z}^8 +e^{\hat{z}})$. To determine the multipliers of the real fixed points, we analyze $h(x)= 10(x+0.99)^2 (9 x^8 +e^{x})  $ in the intervals  $I_1$ and $I_2$. 
\par 
Note that $h'(x)=10(x+0.99)(90 x^8 +2 e^x+xe^x+71.28 x^7+0.99e^x)$. Rewriting it as $h'(x)=10(x+0.99)(9 x^7(10x +7.92) +(x+0.99)e^x+2e ^x)$, it is seen that $h'(x)>0$ for all $x \in (-0.99, -0.792]$. In order to show that $h'(x)>0$ in $J=(-0.792, -0.72]$, we analyze  $p(x)=90 x^8 +71.28 x^7+2e^x >0$ in $J$ and show that $p(x)>0$. 
\par 
We make a repeated use of the following facts about  every continuous function $q: [s,t] \to \mathbb{R}$.
\begin{equation}\label{increasing}
\mbox{If}~  q(t) <0 ~\mbox{ and }~q'(x)>0 ~\mbox{then} ~q(x)<0 ~\mbox{for all }~ x \in [s,t]
 \end{equation}  
 \begin{equation}\label{decreasing}
\mbox{If}~  q(t) >0~\mbox{and }~  q'(x)<0~\mbox{then} ~q(x)>0 ~\mbox{for all }~ x \in [s,t]
 \end{equation} 
Also we need the approximate numerical values of successive derivatives of $p$ at $t=-0.72$.
\\

\begin{tabular}{|c|c|c|c|c|c|c|c|}\hline
$p(t)$ &  $p^{'}(t)$ & $p^{''}(t)$&$p^{'''}(t)$&$p^{(iv)}(t)$&$p^{(v)}(t)$& 
$p^{(vi)}(t)$& $p^{(vii)}(t)$  \\
\hline  
 $0.3$ &  $-1.7$ & $123$&$-1800$&$ 18000$&$-13 \times 10^4 $& 
$6.8 \times 10^5$& $-2.2 \times 10^6$ \\ \hline 
\end{tabular} 
\\
 
Since $p^{(vii)}(t) <0$ and $p^{(viii)}(x)=90(8 !)+2 e^x>0$, $p^{(vii)}(x)<0$ by Equation~\ref{increasing}. Applying Equation~\ref{decreasing} for $q(x)=p^{(vi)}(x)$ we get $p^{(vi)}(x)>0$. Again applying Equation~\ref{increasing} for $q(x)=p^{(v)}(x)$, it is found that $p^{(v)}(x)<0$. Repeating this argument, we finally find that 
$p'(x)<0$ and $p$ is
strictly decreasing in $J$. As $p(t)>0$, $p(x)>0$ for all $x \in J$.

\par Therefore, $h$ is strictly increasing in $(-0.99, -0.72]$. Note that  $0< h(x) < h(-0.72) \approx 0.828$ which means that the multiplier of every fixed point of $f_3$ lying in $[-0.99,-0.72]$ is in $(0,1)$.  Since the fixed point $a_1$ is in $ I_1 \subset (-0.99, -0.72 ] $, $a_1$ is attracting.
\par
Now, rewrite $h'(x)$ as $10(x+0.99)(72 x^7(x+0.99) +2.99 e^x+(xe^x+18 x^8))$ and note that $x e^x + 18 x^8 > 0$.  This is so because $x e^x + 18 x^8 $  is a strictly decreasing function in  $I_2=(-1.069, -1)$ and its minimum value in this interval is positive. It is seen that $h'(x)<0$ for all $x \in I_2$. Thus, $0 < h(x)  < h(-1.069) \approx 0.979$ for all $x \in I_2$ and $a_2$ is attracting. 

\begin{figure}[H]
\centering
\subfloat[Graphs of $ \frac{0.1}{e^x+x^9}-0.99-x$ and the derivative of $ \frac{0.1}{e^x+x^9}-0.99$]
{\includegraphics[width=2.6in,height=2.6in]{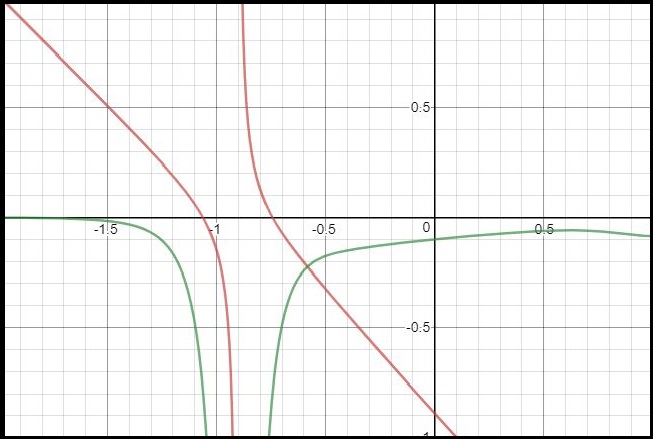}}
 \hspace{0.01cm}
\subfloat[The Fatou set of $ \frac{0.1}{e^z+z^9}-0.99$]
{\includegraphics[width=2.6in,height=2.6in]{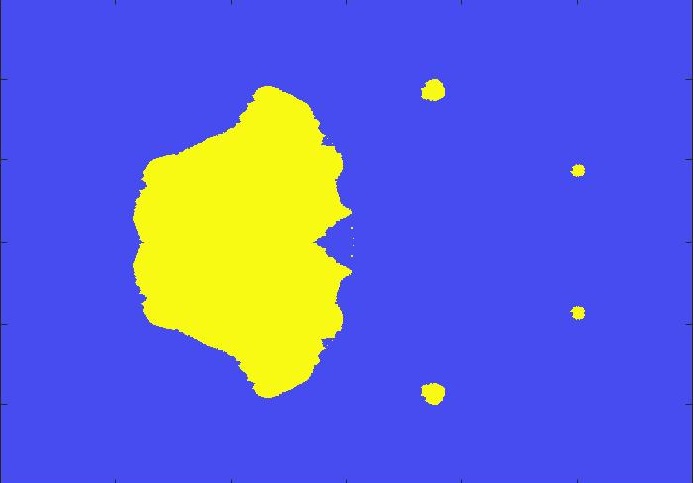}}
\caption{The two attracting domains of $ \frac{0.1}{e^z+z^9}-0.99$}	 
\label{Fatouset-2}
\end{figure}

It is important to note that the bov is in the (immediate)  attracting domain of $a_2$. This is because, $f_{3}$ maps $(-\infty, x_0)$ onto $(-\infty, -0.99)$ which is mapped onto $I=(f_{3}(-0.99), -0.99)$ by $f_3$, where $f_{3}(-0.99) \approx -1.17$. It can be seen numerically that $f_{3}^{7}(-0.99) \approx -1.08, f_{3}^{8}(-0.99) \approx -1.05 $ and  $f_{3}^{7}(I) \subset I $. Further, $f_{3}'(f_{3}^{7}(-0.99) ) \approx -0.613$,  $f_{3}'(f_{3}^{8}(-0.99) ) \approx -0.94$ and therefore $|f_{3}'(x)|<1$ for all $x \in f_{3}^{7}(I)$. Therefore, the attracting domain of $a_2$ is completely invariant. It is shown  blue in  Figure~\ref{Fatouset-2}(b). The attracting domain of $a_1$ is shown  yellow in Figure~\ref{Fatouset-2}(b). 
\end{proof}
\begin{rem}
The attracting domain of $a_1$ contains at least one critical point. This must be non-real as $f_{3}'(x) <0$ for all real $x$. Motivated by the Figure~\ref{Fatouset-2}(b), we conjecture that the Fatou set of $f_3$ is the union of the basins of attractions of these two attracting fixed points, $a_1$ and $a_2$.
\end{rem}
 \subsection{Invariant Siegel disks}
\begin{ex}
	There is a meromorphic function with bov which has an invariant Siegel disk. The Siegel disk is bounded if the bov is stable.
\label{Siegel}\end{ex}
\begin{proof}
Let $\lambda$ be such that $-2 \lambda =e^{i 2 \pi t}$ for a Bryuno number $t$ and consider the function $f_{4}(z)=\lambda(\frac{1}{z+e^z}-1)$. Then $-\lambda$ is the bov of $f_4$ and  $0$ is an irrationally indifferent fixed point  with multiplier $e^{i 2 \pi t}$. Thus, there is an invariant Siegel disk containing $0$ (See Figure~\ref{Fatouset-3}(b)). If the Siegel disk is unbounded then the image of its boundary must accumulate at the bov and therefore, the bov is on the boundary of the Siegel disk. In other words, if the bov is stable then the invariant Siegel disk is bounded. 
\begin{figure}[H]
\centering
\subfloat[Parabolic domain:  $ \lambda =0.5$]
{\includegraphics[width=2.5in,height=2.5in]{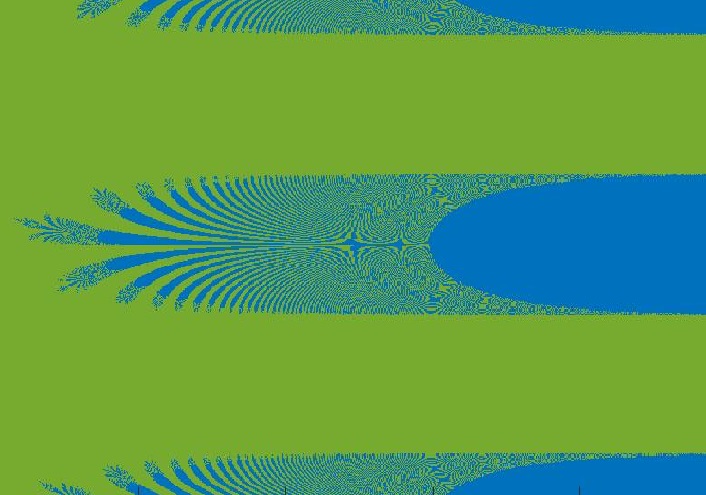}}
\subfloat[Siegel disk: $\lambda=e^{2 \pi i(\frac{\sqrt{5}-1}{2})}$]	{\includegraphics[width=2.5in,height=2.5in]{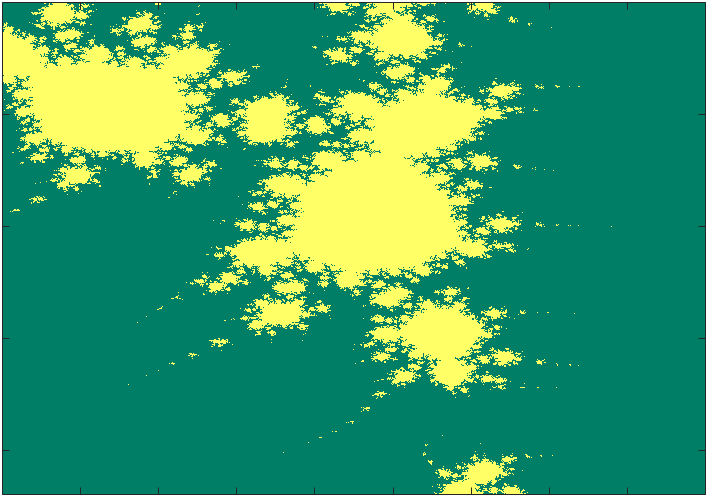}}
\caption{ Parabolic domain and Siegel disk of $\lambda(\frac{1}{z+e^z}-1)$}	 
\label{Fatouset-3}
\end{figure}
\end{proof}
\begin{rem}
For $\lambda =0.5$, the function $\lambda(\frac{1}{z+e^z}-1)$ has an invariant parabolic domain, shown green in Figure~\ref{Fatouset-3}(a). 
\end{rem}

\section{Concluding Remarks}
We conclude with few remarks leading to some directions of research.
\par Lemma~\ref{critically infinite} proves that the bov is always a limit point of the critical values of the function.  A limit point of critical values different from the bov seems to be impossible. Though all the examples discussed in this article confirm this, it remains to be proved.

\par
If the bov is stable then the function has a unique unbounded Fatou component (Lemma~\ref{singleton}). It is yet to be known whether a function  has an unbounded Fatou component when bov is unstable. 
\par
In presence of a stable bov, the Julia components not intersecting the backward orbit of $\infty$ are all bounded and are mapped onto each other. Their topology can be investigated.  Further, assuming that the bov is the only limit point of critical values, an upper bound for the number of periodic cycles of Fatou components may exist. In the absence of recurrent critical points, these issues seem to be tractable.  
\subsection*{Acknowledgment}
This work was initiated during the Fast Track Project (SR-FTP-MS019-2011) funded by the Department of Science and Technology, Govt of India. The second and the third authors were part of it.

\end{document}